\documentclass[parskip=half,bibliography=totoc]{scrartcl}

\usepackage{comment}
\usepackage[automark]{scrlayer-scrpage}								
\usepackage{amsfonts}												
\usepackage{amsmath}												
\usepackage{mathtools}												
\usepackage{stmaryrd}												
\usepackage{amssymb}												
\usepackage{mathrsfs}												
\usepackage{accents}												
\usepackage[T1]{fontenc}											
\usepackage[utf8]{inputenc}											
\usepackage[top=1in,bottom=1in,left=1.25in,right=1.25in]{geometry}	
\usepackage{enumerate} 												
\usepackage{authblk}												
\usepackage{abstract}												
\usepackage{etoolbox}												
\usepackage[normalem]{ulem}

\usepackage{tikz-cd}												

\usepackage{amsthm}													

\usepackage[colorlinks=true,citecolor=blue,allcolors=blue]{hyperref}				
\usepackage[noabbrev]{cleveref}										

\newtheoremstyle{mythm}
{}
{}
{\slshape}
{}
{\bfseries\sffamily}
{.}
{ }
{}
\newtheoremstyle{mydef}
{}
{}
{}
{}
{\bfseries\sffamily}
{.}
{ }
{}

\theoremstyle{mythm}
\newtheorem{thm}{Theorem}[section]

\newtheorem{lem}[thm]{Lemma}
\theoremstyle{mydef}

\allowdisplaybreaks

\clearpairofpagestyles
\ihead[]{\headmark}
\ohead[]{\pagemark}
\deffootnote[1em]{0em}{1em}{%
	\textsuperscript{\thefootnotemark}%
}
\setfootnoterule{3em}

\apptocmd{\sloppy}{\hbadness 10000\relax}{}{}

\title{A note on quaternionic K\"ahler manifolds with ends of finite volume}
\author{V.\ Cort\'es}

\affil{\normalsize $^1$Department of Mathematics\hfill \protect\\  University of Hamburg\hfill \protect\\  
Bundesstra\ss e 55, D-20146 Hamburg, Germany\hfill \protect\\  
\texttt{vicente.cortes@uni-hamburg.de}
}

\date{\today}

\begin{document}
\maketitle

\begin{abstract}
We prove that complete non-locally symmetric quaternionic K\"ahler manifolds with an end of finite volume
exist in all dimensions $4m\ge 4$.

	\par
	\emph{Keywords: quaternionic K\"ahler manifolds, ends of finite volume}\par
	\emph{MSC classification: 53C26}
\end{abstract}

\clearpage
\setcounter{tocdepth}{2}

\section{Introduction}
Quaternionic K\"ahler manifolds constitute one of the most interesting classes of Einstein
manifolds in Riemannian geometry \cite{Bes}. They occur naturally in the context of the classification of 
Riemannian holonomy groups \cite{B}. One of the long-standing conjectures 
in differential geometry is that complete quaternionic K\"ahler manifolds of \emph{positive} scalar curvature 
are symmetric \cite{LS}. At the time of writing, it has been proven in dimensions $4$ \cite{H,FK}, $8$ \cite{PS}, $12$ and $16$ \cite{BWW}. The first examples of complete quaternionic K\"ahler manifolds of \emph{negative} scalar curvature which are not locally symmetric were found in \cite{A2}.  However, until today, no such examples of \emph{finite volume} are known.  

It was shown in \cite{CRT} that complete non-locally symmetric quaternionic K\"ahler manifolds 
of negative scalar curvature with an end of finite volume
exist in dimensions $4$ and~$8$. 

The construction was based on the existence in all dimensions $4m\ge 4$ of complete non-locally symmetric quaternionic K\"ahler manifolds $(M,g)$ of negative scalar curvature with a cohomogeneity one action by a Lie group $G$ of isometries which admits a lattice $\Gamma \subset G$ acting freely an properly discontinuously on $M$.  In fact, the quotients $X=M/\Gamma$, with the induced metric, were shown to have the topological structure of a cylinder $X = \mathbb{R}\times G/\Gamma$ fibering over the line with fibers of finite volume. Moreover, the volume of the half-cylinder $\{ (t,x) \in X\mid t>0 \}\subset X$, was shown to be finite. 
In particular, if $\Gamma \subset G$ is cocompact, then 
$X$ has two ends and one of them is of finite volume. However, cocompact lattices in $G$ were only shown to exist in dimensions $\le 8$. This left the following problem open.

\textbf{Problem.} Do there exist complete non-locally symmetric quaternionic K\"ahler manifolds of negative scalar curvature with an end of finite volume in all dimensions $4m \ge 4$?

In this note we give a positive answer to this problem. 

\begin{thm}\label{main:thm}For all $4m \ge 4$ there exists a complete non-locally symmetric quaternionic K\"ahler manifold $(M,g)$ of negative scalar curvature with two ends such that one end is of finite volume and the other of infinite volume. 
\end{thm}

{\bfseries Acknowledgements}

This work was supported by the German Science Foundation (DFG) under Germany's Excellence Strategy  --  EXC 2121 ``Quantum Universe'' -- 390833306. I am grateful to Daniel Huybrechts for pointing out Lemma~\ref{K3:lem} and its proof. 
I thank Danu Thung and Iv\'an Tulli for helpful comments. 
\section{Proof of Theorem \ref{main:thm}}
\label{2ndSec}
The following lemma is a slight refinement of the statements about the metric (1.8) in \cite[Theorem 1.1]{F}.
For convenience we have redefined the connection forms. For expository reasons we include a proof. 
\begin{lem}\label{metric:lem}Let $(N, g_N)$ be a hyper-K\"ahler manifold with integral K\"ahler forms $\sigma_1$, $\sigma_2$, $\sigma_3$,  
that is the corresponding de Rham classes $[\sigma_i]$ belong to the image of the natural map
$H^2(M,\mathbb{Z})\rightarrow H^2(M,\mathbb{R})$. Let $P\rightarrow N$ be a $T^3$-principal bundle with connection 
$(\alpha_1,\alpha_2,\alpha_3)\in \Omega^1(P,\mathbb{R}^3)$ and curvature forms 
$d\alpha_1 = \sigma_1$, $d\alpha_2 = \sigma_2$, $d\alpha_3 = \sigma_3$, where $T^3 = \mathbb{R}^3/\mathbb{Z}^3$. 
Then 
\begin{itemize}
\item[(i)]
\[ g := dt^2 + 4\rho^2 \sum \alpha_i^2 + \rho g_N,\quad \rho = e^{2t},\] 
is a quaternionic K\"ahler metric on $M := \mathbb{R}\times P$. 
\item[(ii)] The scalar curvature of $(M,g)$ is negative.
\item[(iii)] $(M,g)$ is complete if and only if $(N,g_N)$ is. 
\end{itemize}
\end{lem}
\begin{proof}
(i) 
Consider the two-forms 
\begin{equation} \label{Kf:eq}\omega_i := 2\rho dt\wedge \alpha_i + 4\rho^2 \alpha_j\wedge\alpha_k +  \rho \sigma_i, \quad i=1,2,3, \end{equation}
where $(i,j,k)$ is a cyclic permutation of $\{ 1,2,3\}$. 
Together with the metric $g$ they define an almost hyper-Hermitian structure $(g, J_1,J_2,J_3)$ on $M$ 
with the fundamental forms 
$\omega_i = g \circ J_i = g(J_i\cdot, \cdot )$, $i=1,2,3$. 
To prove that the metric is quaternionic K\"ahler with the 
quaternionic structure $Q= \mathrm{span}\{ J_1, J_2,J_3\}$ it suffices 
to check that the ideal of the exterior algebra generated by the forms $\omega_1, \omega_2, \omega_3$ is 
a differential ideal \cite{S}. This follows immediately by calculating the differentials of (\ref{Kf:eq}), as done in \cite{F}. 

(ii) To see that the scalar curvature is negative it suffices to remark that the fibers of the projection $M \rightarrow N$
are quaternionic submanifolds and of constant negative sectional curvature. Since a quaternionic submanifold has the 
same reduced scalar curvature as the ambient quaternionic K\"ahler manifold \cite{A1}, the claim follows. 

(iii) The completeness of $(N,g_N)$ implies that of $(M,g)$ by observing that 
$g$ is of the form $g=dt^2 + g_t$, where $g_t$ is a family of metrics on $P$ which over compact subsets $K\subset \mathbb{R}$ is uniformly bounded from below by a complete metric $g_K$ on $P$ and applying \cite[Lemma 2]{CHM}. 
In fact, we can simply take the (product) metric $g_K=4\rho_0^2 \sum \alpha_i^2 + \rho_0 g_N$, where 
$\rho_0 = \min_K \rho$. 

For the converse, we note first that any embedded submanifold of a complete 
Riemannian manifold is complete. Hence,  the completeness of $(M,g)$ implies that of $(P,g_t)$. 
Since, for fixed $t$, $g_t$ is simply a product metric, the completeness of the factor $(N,\rho (t) g_N)$ and hence of 
$(N,g_N)$ follows. 
\end{proof}
\begin{lem}\label{K3:lem}
There exists a K3 surface $S$ which admits a hyper-K\"ahler structure\linebreak  
$(g_S,J_1,J_2,J_3)$ with integral K\"ahler forms. 
\end{lem}
\begin{proof}
This follows from \cite[Table 1]{D}. In fact, from the table we see that there exists a  K3 surface of degree 10 with transcendental lattice of rank two and intersection form represented by $10 \cdot \mathrm{id}$ with respect to
a basis of the lattice. Complementing such a basis by the Fubiny study class of the projective embedding,  
we obtain three integral K\"ahler classes. Choosing a representative in each of the three  classes we obtain 
a hyper-K\"ahler triple on $S$ defining the desired hyper-K\"ahler structure. 
\end{proof}
\begin{lem} \label{vol:lem}For every compact hyper-K\"ahler manifold $(N,g_N)$ with integral K\"ahler forms 
the quaternionic K\"ahler manifold $(M,g)$ of Lemma~\ref{metric:lem} has two ends, one of finite volume and the other of infinite volume.  
\end{lem}
\begin{proof}The (metric) volume form of $(M,g)$ is related to the volume form of 
$(P,g_0)$ by 
\[ \mathrm{vol}_g = 8 \rho^{2n+3}dt \wedge \mathrm{vol}_{g_0}.\]
Thus  for all $t_0\le 0\le t_1$ we have 
\[ f(t_0,t_1):=\mathrm{vol}(\{ (t,p) \in M \mid t_0<t<t_1\}) = C \int_{t_0}^{t_1}\rho (t)^{2n+3}dt, \quad C = 8 \int_P \mathrm{vol}_{g_0}.\] 
Inserting $\rho(t) = e^{2t}$ we obtain that 
\begin{eqnarray*} f(0,t) &=& \frac{C}{4n+6} (e^{(4n+6)t}-1) \rightarrow \infty\quad (t\rightarrow \infty)\\
f(-t,0) &=& \frac{C}{4n+6} (1-e^{-(4n+6)t}) \rightarrow \frac{C}{4n+6} \quad (t\rightarrow \infty). 
\end{eqnarray*} 
This proves the claim. 
\end{proof}
\begin{proof} (of Theorem~\ref{main:thm}) 
By  \cite{CRT} we can assume that $m\ge 3$. In fact, it suffices to assume $m\ge 2$.  
Let $(S,g_S)$ be any hyper-K\"ahler manifold as in Lemma~\ref{K3:lem}. 
The product $(N,g_N)=(S,g_S)^{\times (m-1)}$ is a 
hyper-K\"ahler manifold of dimension $4m-4$ with  integral K\"ahler forms. By Lemma~\ref{metric:lem} we can associate
a complete quaternionic K\"ahler manifold $(M,g)$ of dimension $4m$ and by Lemma~\ref{vol:lem} it has an end 
of finite volume. Now it suffices to show that $(M,g)$ is not locally symmetric. Since $(M,g)$ is a complete quaternionic K\"ahler manifold of negative scalar curvature, if it were locally symmetric, 
its universal covering would be a symmetric space of non-compact type and therefore contractible. 
It would follow that $\pi_2(M)=0$ and the homotopy sequence of the fibration $M \rightarrow N$ would then yield 
\[ \pi_2(M)=0 \rightarrow \pi_2(N) \rightarrow \pi_1(\mathbb{R}\times T^3) = \mathbb{Z}^3,\]
but there is no injective homomorphism of $\pi_2(N)$ into $\mathbb{Z}^3$, since $\pi_2(N)=\pi_2(S)^{m-1} = \mathbb{Z}^{22(m-1)}$. This proves that $(M,g)$ is not locally symmetric, establishing Theorem~\ref{main:thm}.  
Here we have used that $S$ is simply connected and therefore, by Hurewicz's theorem, 
$\pi_2(S) \cong H_2(S,\mathbb{Z})$. The latter group is torsion-free\footnote{In fact, the torsion 
of $H_2(S,\mathbb{Z})$ coincides with that of $H^3(S,\mathbb{Z})$ by the universal coefficient theorem
and the latter is trivial, see \cite[Ch. 1]{Hu}.} and its rank is well known to be 22, see \cite[Ch.\ 1]{Hu}.  
So $H_2(S,\mathbb{Z})\cong  \mathbb{Z}^{22}$. 
\end{proof}

\end{document}